\documentclass[11pt, reqno]{amsart}
\usepackage{graphicx, amssymb, amsmath, amsthm}
\usepackage[utf8]{inputenc}
\usepackage{epsfig}
\usepackage{hyperref}
\usepackage{comment}
\usepackage{mathrsfs}
\numberwithin{equation}{section}
\usepackage{subfigure}
\usepackage{tikz}
\usepackage{here}
\usepackage{float}
\usepackage{pifont}
\usepackage{enumitem}
\usetikzlibrary{matrix,arrows}
\usetikzlibrary{shapes}
\usetikzlibrary{calc}
\usetikzlibrary{arrows}
\usetikzlibrary{decorations.pathreplacing,decorations.markings}
\usepackage[all]{xy}

\usepackage{tikz-cd}

\usetikzlibrary{patterns}

\newtheorem{theorem}{Theorem}[section]

\newtheorem{lemma}[theorem]{Lemma}

\newtheorem{proposition}[theorem]{Proposition}
\newtheorem{question}[theorem]{Question}

\theoremstyle{definition}
\newtheorem{definition}[theorem]{Definition}
\newtheorem{remark}[theorem]{Remark}

\makeatletter
\newcommand{\Extend}[5]{\ext@arrow0099{\arrowfill@#1#2#3}{#4}{#5}}
\makeatother

\DeclareMathOperator{\Lip}{Lip}
\DeclareMathOperator{\dist}{dist}
\DeclareMathOperator{\diam}{diam}

\newcommand{\ve}{\varepsilon}
\newcommand{\ti}[1]{\tilde{#1}}
\newcommand{\de}{\partial}

\begin{document}

\title[Interior control]{Interior control for surfaces with positive scalar curvature and its application}

\author[S. Chen]{Shuli Chen}
\address[Shuli Chen]{Department of Mathematics, University of Chicago, 5734 S University Ave, Chicago IL, 60637, United States}
\email{shulichen@uchicago.edu}

\author[J. Chu]{Jianchun Chu}
\address[Jianchun Chu]{School of Mathematical Sciences, Peking University, Yiheyuan Road 5, Beijing 100871, People's Republic of China}
\email{jianchunchu@math.pku.edu.cn}

\author[J. Zhu]{Jintian Zhu}
\address[Jintian Zhu]{Institute for Theoretical Sciences, Westlake University, 600 Dunyu Road, Hangzhou, Zhejiang 310030, People's Republic of China}
\email{zhujintian@westlake.edu.cn}

\renewcommand{\subjclassname}{\textup{2020} Mathematics Subject Classification}

\begin{abstract}
Let $M^{n}$, $n\in\{3,4,5\}$, be a closed aspherical $n$-manifold and $S\subset M$ a subset consisting of disjoint incompressible embedded closed aspherical submanifolds (possibly with different dimensions). When $n =3,4$, we show that $M\setminus S$ cannot admit any complete metric with positive scalar curvature. When $n=5$, we obtain the same result when $S$ contains a submanifold of codimension 1 or 2. The key ingredient is a new interior control for the extrinsic diameter of surfaces with positive scalar curvature.
\end{abstract}

\maketitle

\section{Introduction}\label{Introduction}
Scalar curvature is the weakest curvature invariant of a Riemannian metric, but the existence of metrics with positive scalar curvature
still places strong topological constraints on the underlying manifolds. The well-known
aspherical conjecture asserts that closed aspherical manifolds cannot
admit any smooth metric with positive scalar curvature (see \cite{Ros84,SY87,Gro86}). Recall that a manifold is called \emph{aspherical} if it has contractible universal cover (or equivalently, if the homotopy group $\pi_i$ vanishes for all $i\geq 2$).
Aspherical manifolds are also known as $K(\pi,1)$-spaces or Eilenberg--MacLane spaces in topology.

So far, there have been many progresses on the aspherical conjecture in low dimensions. The two dimensional case follows directly from the Gauss--Bonnet formula combined with the classification of closed surfaces. In 1983, Gromov and Lawson \cite{GL83} verified the three dimensional case with the help of some filling radius estimate. Closely related, it is worth mentioning that Schoen and Yau \cite{SY79c} previously obtained a fundamental group obstruction for 3-manifolds with positive scalar curvature, which in fact can be also used to confirm the three dimensional aspherical conjecture, combined with the final resolution of the virtually Haken conjecture \cite{A13}. In dimension four, the special case with non-zero first Betti number was first confirmed by Wang \cite{Wang19}. We note that, as pointed out in \cite{Wang19}, there exist infinitely many closed aspherical 4-manifolds that are homology 4-spheres \cite{RT05}, which have zero first Betti number. The general case in dimension four was recently established by Chodosh and Li \cite{CL20} based on a previous outline from Schoen and Yau \cite{SY87}. In dimension five, the aspherical conjecture was recently confirmed independently by Chodosh and Li in the same paper \cite{CL20} and also by Gromov in his paper \cite{Gro20} based on the recent development of Gromov's $\mu$-bubble method \cite{Gro19}. Shortly after that, Chodosh, Li and Liokumovich \cite{CLL23} improved the aspherical conjecture into a mapping version in dimensions four and five. Despite these progresses up to dimension five, the aspherical conjecture is still widely open in dimensions greater than five.

With the help of the $\mu$-bubble
method, there have been various efforts to generalize known topological obstructions for positive scalar curvature on closed manifolds to those on non-compact complete manifolds, where the situations are more complicated since the scalar curvature may decay to zero at infinity.
For results on non-compact manifolds obtained from (generalized) connected sum operations from $n$-torus (or even closed Scheon--Yau--Schick manifolds), the readers can refer to \cite{CL20,LUY20,CLSZ21,Chen24}. When the underlying manifold is only aspherical, due to the lack of nonzero homology classes, the $\mu$-bubble method cannot be applied in a direct way to overcome the issue of non-compactness. As a result, not many results are known for non-compact manifolds constructed from closed aspherical ones. In particular,
Gromov \cite[p. 151]{Gro19} asked the following start-up question:
\begin{question}
Are there complete metrics with positive scalar curvature on closed aspherical manifolds with punctures of dimension $4$ and $5$?
\end{question}
In the previous work \cite{CCZ23}, we are able to give a negative answer to Gromov's question. More generally, we considered the question whether there is any complete metric with positive scalar curvature on connected sums of non-compact manifolds with closed aspherical ones, and we proved the following:

\begin{theorem}[{Theorem 1.2 of \cite{CCZ23}}]\label{Thm: connected sum}
Let $N^n$, $n\in\{3,4,5\}$, be a closed aspherical manifold and let $X^n$ be an arbitrary $n$-manifold. Then the connected sum $Y= N \# X$ admits no complete metric with positive scalar curvature.
\end{theorem}

In this paper, we want to generalize Gromov's start-up question by removing other submanifolds beyond punctures. Namely, we consider the following question:

\begin{question}\label{Que: main_original}
Let $M^n$ be a closed aspherical $n$-manifold and $S\subset M$ a closed subset.
For which $S$ can we guarantee that $M\setminus S$ cannot admit any complete metric with positive scalar curvature?
\end{question}

We note that there are several previous results related to this question. In \cite{GL83}, Gromov and Lawson showed that manifolds of the form $\mathbb{T}^n \setminus \mathbb{T}^k$ for $0 \le k < n$ and $n \ge 2$ do not carry a complete metric with positive scalar curvature. In \cite{SWWZ24}, Shi, Wang, Wu, and the third-named author showed that manifolds of the form $M\setminus \Gamma$ do not carry a complete metric with positive scalar curvature, where $M^n$, $3\le n \le 7$ is a closed Schoen--Yau--Schick manifold and $\Gamma$ is a submanifold such that either $\dim\Gamma \le 1$ or the first Betti number $b_1(\Gamma) \le n - 3$. We also note that Question \ref{Que: main_original} is also closely related to the Schoen conjecture regarding uniformly Euclidean ($L^\infty$) metrics up to a blow-up trick \cite{LM19}.

The work in this paper can be considered as a generalization of the above-mentioned result by Gromov and Lawson in \cite{GL83}. In our view, the key feature of the pair $(\mathbb T^n,\mathbb T^k)$ is that the inclusion map $\pi_1(\mathbb T^k)\to \pi_1(\mathbb T^n)$ is injective. Recall that if a submanifold $N$ in $M$ has injective map $\pi_1(N)\to\pi_1(M)$, it is called incompressible in $M$. Under such consideration, we restrict our attention to the case when $S$ consists of incompressible submanifolds, and we propose the following
\begin{question}\label{Que: main}
Let $M^n$ be a closed aspherical $n$-manifold and $S\subset M$ a subset consisting of disjoint incompressible embedded closed aspherical submanifolds (possibly with different dimensions).
Can $M\setminus S$ admit any complete metric with positive scalar curvature?
\end{question}

In this paper, we fully answer Question \ref{Que: main} for $n=3$ and $4$. Namely, we are able to show

\begin{theorem}\label{Thm: main}
Let $M^n$, $n=3,4$, be a closed aspherical $n$-manifold and $S\subset M$ a subset consisting of disjoint incompressible embedded closed aspherical  submanifolds (possibly with different dimensions). Then $M\setminus S$ cannot admit any complete metric with positive scalar curvature. Moreover, if $g$ is a complete smooth metric on $M\setminus S$ with nonnegative scalar curvature, then $g$ is flat, $S$ must consist of hypersurfaces only, and there is a finite cover $(\tilde M,\tilde S)$ of the pair $(M,S)$ such that $\tilde M\setminus \tilde S$ consists of cylinders $\mathbb T^{n-1}\times \mathbb R$.
\end{theorem}

Throughout this paper, we call $(\tilde M,\tilde S)$ a cover of the pair $(M,S)$ if $\tilde M$ is a cover of $M$ associated with the covering map $\pi:\tilde M\to M$ and $\tilde S=\pi^{-1}(S)$. If $\tilde M$ is a finite cover of $M$, then we call $(\tilde M,\tilde S)$ a finite cover of $(M,S)$.

For $n=5$, we can only give a partial answer to Question \ref{Que: main} when $S$ contains a submanifold with codimension $1$ or $2$.

\begin{theorem}\label{Thm: 5d}
Let $M$ be a closed aspherical $5$-manifold and $S\subset M$ a subset consisting of disjoint incompressible closed aspherical embedded submanifolds (possibly with different dimensions).
If $S$ contains a submanifold of codimension 1 or 2, then $M\setminus S$ cannot admit any complete metric with positive scalar curvature. Moreover, if $g$ is a complete smooth metric on $M\setminus S$ with nonnegative scalar curvature, then $g$ is flat, $S$ must consist of hypersurfaces only, and there is a finite cover $(\tilde M,\tilde S)$ of the pair $(M,S)$ such that $\tilde M\setminus \tilde S$ consists of cylinders $\mathbb T^{n-1}\times \mathbb R$.
\end{theorem}
The strategy for the proof of the above theorems essentially follows from the one developed in \cite{CCZ23}, except that some necessary modifications will be made to obtain extrinsic diameter estimates. In particular, the following interior control of the extrinsic diameter plays a key role when $n=3$ and $4$. 

\begin{proposition}\label{Prop: interior}
     Given a positive and continuous function $L:[0,+\infty)\to (0,+\infty)$ there is a positive function $K:(0,+\infty)\to (0,+\infty)$ such that the following property holds: for any $s_{0}>0$, if $(\Sigma,g,\rho)$ is a triple  satisfying
    \begin{itemize}\setlength{\itemsep}{1mm}
        \item $(\Sigma,g)$ is a connected Riemannian surface with non-empty boundary and $\rho:\Sigma\to[0,+\infty)$ is a smooth and proper function with $\Lip\rho <1$;
        \item

        $\rho>K(s_0)$ on $\partial\Sigma$;

        \item  $(\Sigma,g)$ has $\mathbb{T}^*$-stabilized scalar curvature $R\geq L\circ \rho$ in $\Sigma_{K(s_0)}$, where $\mathbb{T}^*$-stabilized scalar curvature is defined in Definition \ref{Defn: stabilized scalar curvature} and we use the notation $\Sigma_s:=\rho^{-1}([0,s])$,
    \end{itemize}
    then we can find finitely many pairwise disjoint disks $\{D_j\}$ in $\Sigma_{K(s_0)}$ such that
    $$\Sigma_{s_0}\subset\bigcup_j D_j$$
    and for each connected component $\Sigma_{s_0}^{o}$ of $\Sigma_{s_0}$, the extrinsic diameter of $\Sigma_{s_0}^{o}$ in $(\Sigma_{K(s_0)},g)$ is bounded by
   $$ 4\pi \left(\min_{[0,K(s_0)]}L\right)^{-\frac{1}{2}}.$$
\end{proposition}

Roughly speaking, we can show that Riemannian surfaces with curvature decay not too fast (bounded from below by a positive function $L$) will have nice extrinsic diameter estimates for interior components independent of the curvature far away (measured by a shielding function $K$). See Figure \ref{Fig: 1} below for the illustration.

\begin{figure}[htbp]
\centering
\includegraphics[width=11cm]{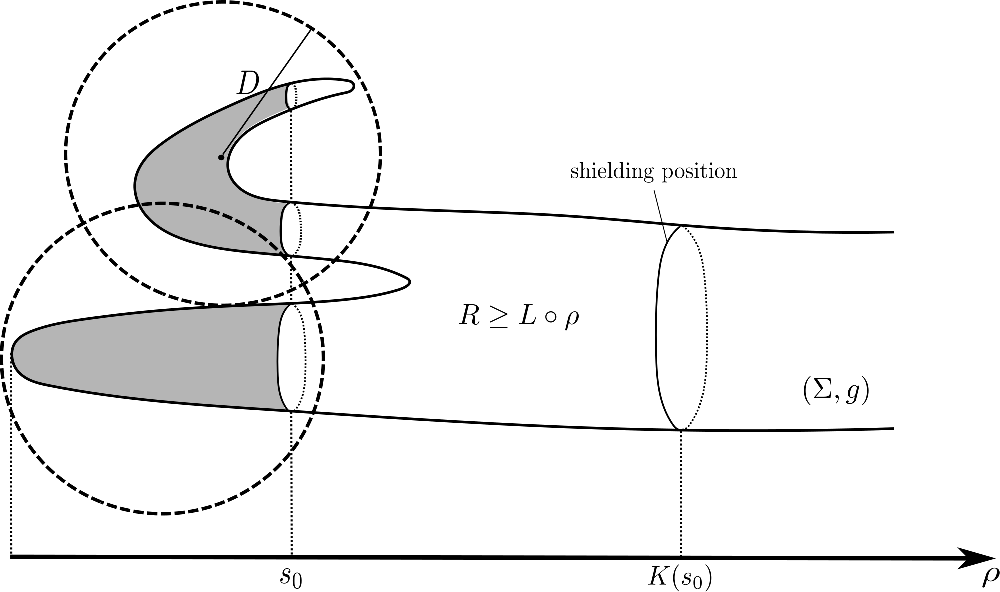}
\caption{Let $(\Sigma,g)$ be a complete Riemannian surface and $\rho$ be a Lipschitz function on $\Sigma$ with $\Lip\rho <1$. If the $\mathbb{T}^*$-stabilized scalar curvature of $\Sigma$ satisfies $R\geq L\circ\rho$ in region $\{\rho \leq K(s_0)\}$, then each component of the region $\{\rho\leq s_0\}$ is contained in some geodesic $D$-ball of $(\Sigma,g)$, where $D$ depends only on the infimum value of $L$ on $[0,K(s_0)]$.}
\label{Fig: 1}
\end{figure}

\begin{remark}
We would like to point out that the topological conclusion (being contained in disjoint disks) of Proposition \ref{Prop: interior} is not used in the proof of Theorem \ref{Thm: main} and \ref{Thm: 5d}. We just include it for completeness and independent interest.
\end{remark}

\subsection{Outline of the proof}
In this subsection, we shall outline the proof for Theorem \ref{Thm: main} and Theorem \ref{Thm: 5d}. For readers' convenience, we first recall the previous strategy for Theorem \ref{Thm: connected sum} from \cite{CCZ23}, and then introduce necessary modifications to be made for Theorem \ref{Thm: main} and Theorem \ref{Thm: 5d}.
\subsubsection{The previous strategy for Theorem \ref{Thm: connected sum}}

First let us briefly recall the strategy developed in \cite{CCZ23} for the proof of Theorem \ref{Thm: connected sum}, which is essentially a relative version of the chain-closing program first proposed in \cite{SY87} and further developed in \cite{CL20}.

For $n \in \{3,4,5\}$, we consider the non-compact connected sum $N\#X$,
where $N$ is a closed aspherical $n$-manifold, and $X$ denotes an arbitrary non-compact $n$-manifold. For simplicity, we denote $N\#X$ by $Y$. Throughout this paper, all homology groups are assumed to have $\mathbb{Z}$ coefficients.

Suppose that $g$ is a complete metric of positive scalar curvature on $Y = N \# X$, then we aim to derive a contradiction. Since $N$ has a contractible universal cover $\tilde N$, we can consider the corresponding cover $\tilde Y=\tilde N\#_{\pi_1(N)}X$ of $Y$. By definition the connected sum  $\tilde Y=\tilde N\#_{\pi_1(N)}X$ induces a natural decomposition $$
\tilde Y=\tilde N_{\ve}\cup \ti X_{\ve}$$
in the following way: for each $X$ we take away a small ball and label it by $\tilde X_{\ve, i}$ with index $i\in \pi_1(N)$, then we can take the union
$$\tilde X_{\ve}=\left(\bigcup_i \tilde X_{\ve,i}\right)$$
such that the complement $\tilde N_{\ve}$ of $\tilde X_{\ve}$ in $\tilde Y$ is diffeomorphic to
$$
\tilde N-\left(\bigcup_i B_i\right).
$$
It is clear that we can guarantee that both $\tilde N_{\ve}$ and $\tilde X_{\ve}$ are $\pi_1(N)$-invariant under the deck transformation of the covering $\tilde Y\to Y$.

The whole proof of Theorem \ref{Thm: connected sum} relies on the following two topological properties of $\tilde Y$:
\begin{itemize}\setlength{\itemsep}{1mm}
\item (Homological property) From the contractibility of $\tilde N$, one can show that the relative homology group $H_k(\tilde Y, \ti X_{\ve})$ is zero for $k = n-1$ and the map $H_k(\tilde Y) \to H_k(\tilde Y, \ti X_{\ve})$ is zero for all $k$.
\item (Quantitative filling property) Since $\ti N_\ve$ is the cover of a compact manifold, for any $k\in \mathbb N_+$ there is a function $F:(0,+\infty)\to (0,+\infty)$ depending only on the triple $(Y, N_\ve, g)$ such that if $C$ is a relative $k$-boundary in $\mathcal B_k(\ti Y, \ti X_{\ve})$ contained in some geodesic ball $B_{r}^{\ti{g}}(\tilde q)$ of $(\ti Y,\tilde g)$ with $\tilde q\in \tilde N_\ve$, then one can find a relative $(k+1)$-chain $\Gamma \in \mathcal Z_{k+1}(\ti Y,  \ti X_{\ve})$ contained in the geodesic ball $B_{F(r)}^{\ti{g}}(\tilde q)$ of $(\tilde Y,\tilde g)$ with $\partial\Gamma=C$ modulo $ \ti X_{\ve}$.
\end{itemize}

The basic idea is to complete a chain-closing program in $\tilde Y$. Namely, we fix a proper line $\tilde \sigma$ in $\tilde{N}_\ve$ and take the boundary of a tubular neighborhood $\mathcal N$ of $\tilde \sigma|_{[0,+\infty)}$. From our construction $\partial \mathcal N$ is a locally-finite relative chain having non-zero intersection number with $\tilde\sigma$. Due to the homology property of $\tilde Y$, if one can close this chain without creating further intersections, then the newly obtained relative cycle represents a non-trivial relative homology class in $H_{n-1}(\tilde Y, \ti X_{\ve})$, which leads to a desired contradiction.

To be more precise, through cutting $\partial \mathcal N$ at a finite length we can obtain a hypersurface $\tilde M_{n-1}$ with boundary in $\tilde Y$, which has non-zero algebraic intersection with $\tilde\sigma$.
In this step, we use the fact $H_{n-1}(\tilde Y, \ti X_{\ve}) = 0$. Let $M_{n-1}$ be the minimizing hypersurface obtained from solving the Plateau problem with the prescribed boundary $\partial\tilde M_{n-1}$. After using $\mu$-bubble method we can further construct a codimension-two closed submanifold $M_{n-2}$ in $\tilde Y$. Through adjusting the size of the tubular neighborhood as well as the cutting length we can guarantee
\begin{itemize}\setlength{\itemsep}{1mm}
\item $M_{n-2}$ and $\partial M_{n-1}$ enclose a bounded region $\Omega_{n-1}$ in $M_{n-1}$ which is disjoint from $\tilde\sigma$;
\item $\dist(M_{n-2},\tilde \sigma)>L$ for arbitrarily large $L>0$;
\item $M_{n-2}$ has $\mathbb{T}^{*}$-stabilized scalar curvature no less than $R(g)-\mu_{\mathrm{loss}}$ for arbitrarily small $\mu_{\mathrm{loss}}>0$.
\end{itemize}
The last property above guarantees that the $\mathbb{T}^{*}$-stabilized scalar curvature on $M_{n-2} \cap \ti N_\ve$ has a positive lower bound $\underline{R}$. The next step is to fill $M_{n-2}$ with chain without creating further intersections with $\tilde \sigma$ and we have to make discussion according to the dimension $n$.

In dimensions 3 and 4, we have the following two properties:
\begin{enumerate}
\item Note that $M_{n-2}$ is either a curve or a surface, so the scalar curvature lower bound $\underline R$ guarantees that the surface $M_{n-2}  \cap \ti N_\ve$ is within a distance $C(\underline{R})$ from its boundary.
\item
After enlarging $\tilde N_{\ve}$ we can guarantee that the ends $\tilde X_{\ve,i}$ are at a far enough distance away from each other. Then the surface $M_{n-2}$ can only enter one $\ti X_{\ve, i}$. Since $\partial \ti X_{\ve, i}$ is closed, we can bound the \emph{extrinsic diameter} of the boundary of $M_{n-2} \cap \ti N_\ve$ by the extrinsic diameter of $\partial \ti X_{\ve, i}$, which is a constant independent of $i$.
\end{enumerate}
Combining these two points, we can bound the extrinsic diameter of $M_{n-2} \cap \ti N_\ve$ by $C(\underline{R}) + \diam (\partial \tilde X_{\ve,i})$.

Once we obtain the above extrinsic diameter estimate, we can then apply the quantitative filling property and fill in $M_{n-2}$ by some relative $(n-1)$-chain $\Gamma$ within some fixed distance independent of $L$. When $L$ is large this chain will have no intersection with $\ti \sigma$. Then we can form the relative $(n-1)$-cycle $$\tilde \Gamma_{n-1}:=\tilde M_{n-1}+\Omega_{n-1}+\Gamma.$$
This relative cycle has non-zero algebraic intersection number with the line $\ti \sigma$, which contradicts $H_{n-1}(\tilde Y, \ti X_{\ve}) = 0$.

In dimension 5, since $M_{n-2}$ is a 3-manifold, the property (1) no longer holds for $M_3$. In order to deal with this issue we have to make a refinement of Chodosh--Li's slice-and-dice argument in \cite{CL20}, to which the most effort in \cite{CCZ23} is devoted.
The idea for the original slice-and-dice argument of Chodosh-Li in the closed case  is as follows: Given the closed 3-manifold $M_3$ with $\mathbb{T}^{*}$-stabilized positive scalar curvature, one can first use closed minimal surfaces $S_k$ to slice $M_3$ into a tree-like manifold with simpler second homology group, and then one can use free boundary $\mu$-bubbles $D_l$ to further dice the tree-like manifold into blocks.
The positive lower bound of the $\mathbb{T}^{*}$-stabilized positive scalar curvature of $M_3$ can be used to bound the intrinsic diameter of the slicing surfaces $S_k$ and the dicing surfaces $D_l$, and therefore one can bound the diameter of the blocks.
In our case, however, we can only obtain an inradius bound for $S_k \cap \ti N_\ve$ and $D_l \cap \ti N_\ve$ from the positive lower bound of the $\mathbb{T}^{*}$-stabilized positive scalar curvature of $M_3\cap \tilde N_{\ve}$ due to the existence of possibly mean-concave boundary. Roughly speaking, the key observation is that
we can establish analogies of properties (1) and (2) for surfaces $S_k$ and $D_l$,
which gives an \emph{extrinsic diameter} bound of the surfaces $S_k \cap \ti N_\ve$ and $D_l \cap \ti N_\ve$, and thereby we can still bound the extrinsic diameter of the blocks. A similar filling argument based on the quantitative filling property gives the desired contradiction again.

\subsubsection{The modification}
In the following discussion, let $M^n$, $n\in\{3,4,5\}$, be a closed aspherical $n$-manifold and $S\subset M$ a subset consisting of disjoint incompressible embedded closed aspherical   submanifolds (possibly with different dimensions). We assume that $(M\setminus S, g)$ is a complete manifold with positive scalar curvature and aim to derive a contradiction.

We follow the same strategy as above.
Let $\pi:\tilde M\to M$ be the universal covering of $M$. Take $\tilde S=\pi^{-1}(S)$ and $\tilde g=\pi^*g$.  Since $S$ consists of incompressible aspherical submanifolds in $M$, its lift $\tilde S$ consists of contractible submanifolds in $\tilde M$. Take $S_\ve$ to be a tubular neighborhood of $S$ in $M$ and denote $\tilde S_\ve=\pi^{-1}(S_\ve)$. We consider the manifold $(\tilde M \setminus \tilde S, \ti g)$, and the relative homology groups $H_k(\tilde M \setminus \tilde S, \ti S_\ve)$.
Since the components of $\tilde S$ are contractible, $\tilde M \setminus \tilde S$ still satisfy the same topological properties as $\tilde Y$.
Then we can take a proper line $\tilde \sigma$ in $\tilde M\setminus \tilde S_\ve$ and find $M_{n-2}$ the same way. If we can fill in $M_{n-2}$ by some relative $(n-1)$-chain within some fixed distance then we are done by the same arguments as above.

The difference here is that we don't have the property (2) anymore because now each component of $\partial \ti S_\ve$ is no longer compact. So we need to find another way to control the extrinsic diameter of $M_{n-2} \cap (\tilde M\setminus \tilde S_\ve)$ such that the filling argument can run smoothly.
In dimensions 3 and 4, recall that we have the inradius estimate for $M_{n-2}\cap (\tilde M\setminus\tilde S_{\ve})$. When $n=3$, the submanifold $M_{n-2}$ consists of curves and the inradius estimate gives the diameter estimate for each component of $M_{n-2}\cap (\tilde M\setminus \tilde S_{\ve})$. Then we can fill $M_{n-2}\cap (\tilde M\setminus \tilde S_{\ve})$ with relative chains of bounded size, component by component. When $n=4$, we take the same idea and try to bound \emph{the extrinsic diameter of each component} of $M_{n-2} \cap (\tilde M\setminus \tilde S_\ve)$. Since the $\mathbb{T}^{*}$-stabilized scalar curvature on $M_{n-2} \cap \pi^{-1}(K)$ has a positive lower bound for any compact subset $K$ of $M\setminus S$, the worrying case is that the surface $M_{n-2}$ stretches to infinity. Fortunately, since $M_{n-2}$ is a surface, this issue can be handled with the help of Proposition \ref{Prop: interior}, which bounds the extrinsic diameter of every interior region of a complete surface with positive $\mathbb{T}^{*}$-stabilized scalar curvature.

However, this type of estimate no longer holds for 3-manifolds. Moreover, if we run the slice-and-dice argument in dimension 5, for the slicing surfaces $S_{k}$ and dicing surfaces $D_l$, we can only bound the extrinsic diameter of each component of $S_{k} \cap (\tilde M\setminus \tilde S_\ve)$ and $D_{l} \cap (\tilde M\setminus \tilde S_\ve)$,
which is not enough to obtain the extrinsic diameter estimate for blocks.
As a result, we can only partially answer Question \ref{Que: main} in dimension 5 when $S$ contains a submanifold with codimension $1$ or $2$.

In order to obtain the partial result --- Theorem \ref{Thm: 5d}, we make the following observation. If $S$ contains one submanifold of codimension 1 or 2, then this essentially implies $M\setminus S$ has a closed incompressible aspherical hypersurface. The presence of such a hypersurface already forbids any metric with positive scalar curvature on $M\setminus S$ from \cite[Theorem 1.1]{CLSZ21}.

\bigskip

The organization of this paper is as follows. In Section \ref{Sec: extrinsic}, we prove the key extrinsic diameter estimate, Proposition \ref{Prop: interior}.
In Section \ref{Sec: proof of main}, we prove our main result in dimension 3 and 4, Theorem \ref{Thm: main}. In Section \ref{Sec: proof of 5d}, we prove
our result in dimension 5, Theorem \ref{Thm: 5d}.

\subsection*{Acknowledgments}
We thank Professors Otis Chodosh, Antoine Song, Ailana Fraser, Kai Xu, and Rudolf Zeidler for their interest and helpful suggestions. Part of this research was performed while the first-named author was in residence at the Simons Laufer Mathematical Sciences Institute (formerly MSRI), which is supported by the National Science Foundation (Grant No. DMS-1928930). The second-named author was partially supported by National Key R\&D Program of China 2024YFA1014800 and 2023YFA1009900, and NSFC grants 12471052 and 12271008. The third-named author was partially supported by National Key R\&D Program of China 2023YFA1009900 and NSFC grant 12401072 as well as the start-up fund from Westlake University.

\section{Extrinsic diameter estimate}
\label{Sec: extrinsic}
We begin with some useful definitions and results.
\begin{definition}\label{Defn: stabilized scalar curvature}
Let $(M,g)$ be a Riemannian manifold possibly with non-empty boundary and $R$ a smooth function on $M$. We say that $(M,g)$ has $\mathbb T^l$-stabilized scalar curvature $R$ if there are $l$ positive smooth functions $v_1,\ldots,v_l$ on $M$ such that the scalar curvature $\tilde R$ of $(M\times \mathbb T^l,g+\sum_i v_i^2\mathrm d\theta_i^2)$ is the $\mathbb T^l$-invariant extension of $R$.
We say that $(M,g)$ has $\mathbb T^*$-stabilized scalar curvature $R$ if $(M,g)$ has $\mathbb T^l$-stabilized scalar curvature $R$ for some positive integer $l$.
\end{definition}

\begin{definition}
We say that $(M,\partial_\pm,g)$ is a Riemannian band if $(M,g)$ is a compact Riemannian manifold with boundary, and $\partial_+$ and $\partial_-$ are two disjoint non-empty smooth compact portions of $\partial M$ such that $\overline{\partial M-(\partial_+\cup \partial_-)}$ is also a compact smooth portion of $\partial M$.
\end{definition}

\begin{lemma}[Lemma 2.12 of \cite{SWWZ24}]\label{Lem: D estimate}
Given any positive constant $c_0$, there is a positive constant $D_{0}=D_0(c_0)$ such that the following holds.
Let $(M^{n},\partial_\pm,g)$ be a Riemannian band with $n\leq7$ and $\overline{\partial M-(\partial_+\cup\partial_-)}=\emptyset$. If
\[
\rho: (M,\de_{\pm}) \to ([-D,D],\pm D)
\]
is a smooth function satisfying
\begin{itemize}\setlength{\itemsep}{1mm}
\item $\Lip\rho<1$;
\item $(M,g)$ has $\mathbb T^*$-stabilized scalar curvature $R$ such that $R\geq c_{0}$ in $\rho^{-1}([-1,1])$ and $R\geq0$ in $M$;
\item Any closed hypersurface $\Sigma$ in $M$ separating $\de_{-}$ from $\de_{+}$ does not admit any metric with positive $\mathbb{T}^{*}$-stabilized scalar curvature,
\end{itemize}
then $D\leq D_{0}$.
\end{lemma}

\begin{lemma}[Lemma 2.14 of \cite{CCZ23}]\label{Lem: free boundary}
Let $(M^{n},\partial_\pm,g)$ be a Riemannian band with $2\leq n\leq 7$ and $\dist(\partial_+,\partial_-)>2d_0$. If $\Gamma:=\overline{\partial M-(\partial_+\cup\partial_-)}$ does not intersect $\de_{\pm}$ and $(M,g)$ has $\mathbb T^l$-stabilized scalar curvature $R$, then there exists a hypersurface $\Sigma$ (possibly with non-empty boundary $\partial\Sigma$) intersecting $\Gamma$ orthogonally such that
\begin{itemize}\setlength{\itemsep}{1mm}
\item $\Sigma$ and $\partial_-$ bound a relative region $\Omega$ relative to $\Gamma$, i.e. $\partial\Omega-(\Sigma\cup \partial_-)\subset \Gamma$;
\item $(\Sigma,g)$ has $\mathbb T^{l+1}$-stabilized scalar curvature $R'$ where $R'\geq R-\pi^2/d_0^2$.
\end{itemize}
\end{lemma}

We can then give a proof of Proposition \ref{Prop: interior}.

\begin{proof}[Proof of Proposition \ref{Prop: interior}]
It suffices to determine the value of $K(s_{0})$ for $s_{0}$. Set
\[
c_{0} = \min_{[0,s_{0}+2]}L,\ \
s_{1} = s_{0}+D_{0}(c_{0})+D_{0}(3c_{0}/4)+2,
\]
\[
c_1=\min_{[0,s_{1}+3]}L,\ \ K(s_{0})=s_1+D_0(c_1)+3,
\]
where $D_0(c_0)$, $D_0(3c_0/4)$ and $D_0(c_1)$ are the constants defined in Lemma \ref{Lem: D estimate}.
Let $s'$ be a regular value of $\rho$ in $(K(s_0)-1,K(s_0))$.
We deal with interior topology and interior geometry separately.

\bigskip

\noindent
{\bf I. Interior topology.}
The argument is broken into three steps.

\medskip

\noindent
{\bf Step 1.} Let $s_{1}^{*}$ be a regular value of $\rho$ in $(s_{1},s_{1}+1)$. Then each component $\Sigma_{s_1^*}^o$ of $\Sigma_{s_1^*}$ is homeomorphic to a $2$-sphere with finitely many disks removed.

\medskip

Otherwise, some component $\Sigma_{s_1^*}^o$ must be $m\mathbb{T}^{2}$ or $m\mathbb{RP}^{2}$ after filling with disks. In particular, $\Sigma_{s_1^*}^o$ has the form of $\mathbb{T}^2\#S$ or $\mathbb{RP}^2\#S$ for some surface $S$, then we can find $\hat \Sigma^o_{s_1^*}\subset \Sigma_{s_1^*}^o$ homeomorphic to $\mathbb T^2-B$ or $\mathbb{RP}^{2}-B$.
Let $\Sigma_{s'}^o$ denote the component of $\Sigma_{s'}$ containing $\Sigma_{s_1^*}^o$, and let $S_+$ and $S_-$ denote copies of $\Sigma_{s'}^o\setminus \hat\Sigma_{s_1^*}^o$. Consider a double cover $\tilde\Sigma_{s_1^*}^o$ of $\hat\Sigma_{s_1^*}^o$ and clearly $\tilde\Sigma_{s_1^*}^o$ has the form of $\mathbb T^2-(B_+\sqcup B_-)$ or $\mathbb S^2-(B_+\sqcup B_-)$, where $B_+$ and $B_-$ are lifts of $B$. Correspondingly, the surface $\Sigma_{s'}^o$ is lifted to some double cover
\[
\tilde\Sigma_{s'}^o=S_+\cup \tilde\Sigma_{s_1^*}^o \cup S_-.
\]
Denote $\pi:\tilde\Sigma_{s'}^o\to \Sigma_{s'}^o$ to be the covering map. Then the map
\begin{equation*}
\tilde \rho(x)=\left\{
\begin{array}{cc}
(\rho\circ \pi-s_1^*)_+,& x\in S_+;\\[1mm]
-(\rho\circ \pi-s_1^*)_+,& x\in S_-;\\[1mm]
0,&\mbox{otherwise},
\end{array}\right.
\end{equation*}
is a Lipschitz function on $\tilde\Sigma_{s'}^o$ with $\Lip \tilde \rho<1$. Also we see that the function $\tilde \rho $ has range $[s_1^*-s',s'-s_1^*]$, and that $\tilde \Sigma_{s'}^o$ has $\mathbb{T}^*$-stabilized scalar curvature $R\geq c_1\chi_{[-2,2]}\circ\tilde \rho$. We further notice that the third hypothesis of Lemma \ref{Lem: D estimate} is trivial on surfaces since $\mathbb{S}^1$ does not admit any metric with positive $\mathbb{T}^{*}$-stabilized scalar curvature. Thus by smooth approximation and Lemma \ref{Lem: D estimate} we obtain
$$
D_0(c_1)+1 = K(s_{0})-1-(s_{1}+1) < s'-s_1^*\leq D_0(c_1),
$$
which leads to a contradiction.

\medskip

\noindent
{\bf Step 2.} Let $s_0^*$ be a regular value of $\rho$ in $(s_0,s_0+1)$. Then each circle $\gamma$ in $\rho^{-1}(s_0^*)$ bounds a disk in $\Sigma_{s_{1}^{*}}$.

\medskip

Let $\Sigma_{s_1^*}^o$ be the component of $\Sigma_{s_1^*}$ containing $\gamma$. Suppose that $\gamma$ does not bound a disk in $\Sigma_{s_1^*}$. Recall that $\Sigma_{s_1^*}^{o}$ is homeomorphic to a $2$-sphere with finitely disks removed. From this we conclude that $\gamma$ separates the boundary components of $\partial\Sigma_{s_1^*}^o$ into two non-empty collections. Denote
$$\Sigma_{s_1^*}^o-\gamma=\Omega_+\sqcup \Omega_-.$$
Define
\begin{equation*}
\rho_{s_1^*}(x)=\left\{
\begin{array}{cc}
\min\{\dist(x,\gamma), s_1^*-s_0^*\},& x\in \Omega_+;\\[1mm]
-\min\{\dist(x,\gamma), s_1^*-s_0^*\},& x\in \Omega_-;\\[1mm]
0,&\mbox{otherwise}.
\end{array}\right.
\end{equation*}
Then $\rho_{s_1^*}$ is a Lipschitz function on $\Sigma_{s_1^*}^o$ with $\Lip \rho_{s_1^*}\leq1$. Also $\rho_{s_1^*}$ has range $[s_0^*-s_1^*,s_1^*-s_0^*]$, and that $\Sigma_{s_1^*}^o$ has $\mathbb{T}^*$-stabilized scalar curvature $$R\geq c_0\chi_{[-s_0-2+s_0^*,s_0+2-s_0^*]}\circ \rho_{s_1^*}.$$
By smooth approximation and Lemma \ref{Lem: D estimate} we obtain
$$
D_0(c_0)+D_0(3c_0/4)+1 = s_{1}-s_{0}-1 < s_1^*-s_0^*\leq D_0(c_0),
$$
which leads to a contradiction.

\medskip

\noindent
{\bf Step 3.} $\Sigma_{s_0}$ is contained in $\sqcup_j D_j$, where $\{D_j\}$ are finitely many disjoint disks in $\Sigma_{s_1^*}$.

\medskip

By passing to each component we can assume $\Sigma_{s_1^*}$ to be connected. Due to the fact $\Sigma_{s_0}\subset \Sigma_{s_0^*}$ it suffices to show that $\Sigma_{s_0^*}$ is contained in $\sqcup_j D_j$, where $\{D_j\}$ are finitely many disjoint disks in $\Sigma_{s_1^*}$. From our assumption we know that $\Sigma$ is connected and so $\Sigma_{s'}$ has non-empty boundary. In particular, $\Sigma_{s_1^*}$ has non-empty boundary as well. Denote $\rho^{-1}(s_0^*)=\sqcup_k\gamma_k$. From the previous steps we conclude that each $\gamma_k$ bounds a unique disk $D_k$ in $\Sigma_{s_1^*}$.

We claim that for $k\neq l$ one of the following happens:
\begin{itemize}\setlength{\itemsep}{1mm}
\item $D_k$ and $D_l$ are disjoint;
\item $D_k\subset D_l$ or $D_l\subset D_k$.
\end{itemize}
Otherwise, we can find a point $p_{kl}\in D_k\cap D_l$ and a pair of points $p_k\in D_k\setminus D_l$ and $p_l\in D_l\setminus D_k$. Connecting $p_{kl}$ and $p_k$ with path in $D_k$ we see $\gamma_l\cap D_k\neq\emptyset$. Since $\gamma_k$ and $\gamma_l$ are disjoint, we conclude $\gamma_l\subset D_k$ and similarly $\gamma_k\subset D_l$. In particular, the union $D_k\cup D_l$ is a closed surface in $\Sigma_{s_1^*}$, which is impossible since $\Sigma_{s_1^*}$ is connected with non-empty boundary.

Now we can take the collection of maximal disks $\{D_{j,\mathrm{max}}\}$ with respect to the inclusion. From previous claim we conclude that $D_{j,\mathrm{max}}$ are pairwise disjoint.

It remains to show $\Sigma_{s_0^*}\subset \sqcup D_{j,\mathrm{max}}$. Take any component $\Sigma_{s_0^*}^o$ of $\Sigma_{s_0^*}$. Since $\Sigma_{s_0^*}^o$ does not cross $\sqcup_k\gamma_k$, either $\Sigma_{s_0^*}^o\subset \sqcup D_{j,\mathrm{max}}$ or $\Sigma_{s_0^*}^o$ is disjoint with all disks $D_k$. Denote $\partial\Sigma_{s_0^*}^o=\sqcup_l \gamma_{k_l}$. Then the union
$$
\Sigma_{s_0^*}^o\cup\left(\bigsqcup_lD_{k_l}\right)
$$
is a closed surface in $\Sigma_{s_1^*}$, which leads to the same contradiction as before.

\bigskip
\noindent
{\bf II. Interior geometry.}
To establish the extrinsic diameter estimate, we argue by contradiction. Set
\[
c_2=\min_{[0,K(s_{0})]}L.
\]
If
\[
\diam(\Sigma_{s_0}^{o} \subset (\Sigma_{s_1^*},g)) > 4\pi \left(\min_{[0,K(s_{0})]}L\right)^{-\frac{1}{2}} = \frac{4\pi}{\sqrt{c_{2}}},
\]
then there exist two points $p_{+},p_{-}\in\Sigma_{s_0}^{o}$ such that $\dist(p_{+},p_{-})>4\pi/\sqrt{c_{2}}$. Recall that $s'$ is a regular value of $\rho$ in $(K(s_{0})-1,K(s_{0}))$. For sufficiently small $\ve>0$, set
\[
M = \Sigma_{s'}\setminus(B_{\ve}^{g}(p_{-})\cup B_{\ve}^{g}(p_{+})), \ \
\de_{-} = \de B_{\ve}^{g}(p_{-}), \ \
\de_{+} = \de B_{\ve}^{g}(p_{+}), \ \
\Gamma = \de\Sigma_{s'},
\]
and then $(M,\de_{\pm},g)$ is a Riemannian band with $\dist(\partial_+,\partial_-)>4\pi/\sqrt{c_{2}}$ and $\Gamma\cap(\de_{+}\cup\de_{-})=\emptyset$. Recall that $(\Sigma,g)$ has $\mathbb{T}^l$-stabilized scalar curvature $R\geq L\circ \rho$ in $\Sigma_{K(s_0)}$ and so in $\Sigma_{s'}$.
By Lemma \ref{Lem: free boundary} (with the choice $d_{0}=2\pi/\sqrt{c_{2}}$), there exists a curve $\gamma$ satisfying
\begin{itemize}\setlength{\itemsep}{1mm}
\item[(a)] $\gamma$ and $\partial_-$ bound a relative region $\Omega$ relative to $\Gamma$;
\item[(b)] $\gamma$ has the positive $\mathbb T^{l+1}$-stabilized scalar curvature
\[
R' \geq R-\frac{c_{2}}{4} \geq L\circ \rho-\frac{c_{2}}{4} \geq c_{2}-\frac{c_{2}}{4} = \frac{3}{4}c_{2} > 0.
\]
\end{itemize}
We split the argument into two cases:

\medskip

{\it Case 1. $\gamma\cap\rho^{-1}(s')=\emptyset$.} In this case, $\de\gamma=\emptyset$ and each component of $\gamma$ is $\mathbb{S}^{1}$. Then (b) shows that $\mathbb{S}^{1}$ has positive $\mathbb T^{l+1}$-stabilized scalar curvature. This contradicts the fact that $\mathbb{T}^{n}$ cannot admit any smooth metric with positive scalar curvature.

\medskip

{\it Case 2. $\gamma\cap\rho^{-1}(s')\neq\emptyset$.} Let $\sigma$ be a path in $\Sigma_{s_{0}}^{o}$ such that
\begin{itemize}\setlength{\itemsep}{1mm}
\item $\sigma$ connects $p_{-}$ and $p_{+}$;
\item $\sigma$ has non-zero algebraic intersection with $\de_{-}$.
\end{itemize}
Using (a), $\gamma$ is relative homologous to $\de_{-}$ relative to $\Gamma$ and then $\gamma$ has non-zero algebraic intersection with $\sigma$. In particular, we obtain $\gamma\cap\sigma\neq\emptyset$ and so $\gamma\cap\Sigma_{s_{0}}\neq\emptyset$. Fix a point $x_{0}\in\gamma\cap\Sigma_{s_{0}}$ and the component of $\gamma$ containing $x_{0}$. For convenience, we still denote this component by $\gamma$. Choose an arc length parametrization of $\gamma$ such that
\[
\gamma:[-D',D'']\to M, \ \
\gamma(0) = x_{0}, \ \
\gamma(-D'),\,\gamma(D'') \in \rho^{-1}(s'),
\]
where $D'$ and $D''$ are two positive constants. We next assume that $D'\geq D''$, as the case $D''>D'$ can be proved similarly. Since $\Lip\rho <1$, then
\begin{equation}\label{rho gamma t control}
|\rho(\gamma(t))-\rho(\gamma(0))| < |t| \ \text{for $t\in[-D',D''] ,\, t \neq 0$}.
\end{equation}
Choosing $t=D''$ and recalling the facts $\gamma(D'') \in \rho^{-1}(s')$ and $\rho(\gamma(0))\leq s_0$, we can derive from \eqref{rho gamma t control} that
\begin{equation}\label{D'' lower bound}
D'' > s' - s_{0} > K(s_{0})-2-s_{0} = D_{0}(c_{0})+D_{0}(3c_{0}/4)+D_0(c_1)+3.
\end{equation}
Using \eqref{rho gamma t control} again, we obtain
\[
\rho(\gamma(t)) < s_{0}+1 \ \text{for $t\in[-1,1]$},
\]
which implies $\gamma([-1,1])\subset\Sigma_{s_{0}+1}$. Together with (b) and $c_{2}\leq c_{0}$, we know that $\gamma$ has positive $\mathbb T^{l+1}$-stabilized scalar curvature $R'$ with
\[
R' \geq R-\frac{c_{2}}{4} \geq L\circ \rho-\frac{c_{2}}{4} \geq c_{0}-\frac{c_{2}}{4} \geq \frac{3c_{0}}{4} > 0 \ \text{in $\gamma([-1,1])$}.
\]
Applying Lemma \ref{Lem: D estimate} to $\gamma([-D'',D''])$ and $\rho|_{\gamma([-D'',D''])}$, we obtain $D''\leq D_{0}(3c_{0}/4)$, which contradicts with \eqref{D'' lower bound}.
\end{proof}

\section{Proof of Theorem \ref{Thm: main}}
In the rest of this paper, we define {\it the orientable cover} of a manifold $N$ to be
\begin{itemize}\setlength{\itemsep}{1mm}
\item $N$ when $N$ is orientable;
\item the orientation double cover of $N$ when $N$ is non-orientable.
\end{itemize}
\label{Sec: proof of main}
\begin{proof}[Proof of Theorem \ref{Thm: main}]
By passing to the orientable cover of $M$, we can always assume that $M$ is orientable. We just need to deal with the case when $n=4$ since we can make Riemannian product with $\mathbb S^1$ in dimension three. First we take an embedded closed curve $\sigma$ disjoint from $S$ such that $[\sigma]\neq 0\in \pi_1(M)$, which can be done as follows.

\medskip

{\it Case 1. $S$ contains some hypersurface $S_0$.} Then we can know $\pi_1(S_0)\neq 0$ from its aspherical property. Take an embedded circle $\sigma_0$ in $S_0$ with $[\sigma_0]\neq 0\in \pi_1(S_0)$. If $\sigma_0^*(NS_0)$ is a trivial bundle, we pick a nowhere-zero normal vector field along $\sigma_{0}$ and push $\sigma_0$ away from $S_{0}$ to obtain an embedded closed curve $\sigma$. Otherwise, we do the same thing for the double cover of $\sigma_0$ to obtain an embedded closed curve $\sigma$. From the construction we know that $\sigma$ is homotopic to $\sigma_0$ or $\sigma_0^2$. Since $M$ is aspherical, $\pi_1(M)$ has no torsion. Since the map $\pi_1(S_0)\to\pi_1(M)$ is injective, we see $[\sigma]\neq 0\in\pi_1(M)$.

\medskip

{\it Case 2. $S$ contains no hypersurface.} Then we just take an embedded closed curve $\sigma$ in $M$ with $[\sigma]\neq 0\in \pi_1(M)$. Since $S$ has codimension greater than one, we can perturb $\sigma$ such that $\sigma$ is disjoint from $S$.

\medskip

In the following, we are going to deduce a contradiction assuming that there is a complete metric $g$ on $M\setminus S$ with positive scalar curvature. We begin with the basic set-up. Let $\pi:\tilde M\to M$ denote the universal covering of $M$. Take $\tilde S=\pi^{-1}(S)$ and $\tilde g=\pi^*g$. Since $S$ consists of incompressible aspherical submanifolds in $M$, its lift $\tilde S$ consists of contractible submanifolds in $\tilde M$. On the other hand, we take a line $\tilde \sigma$ in $\tilde M$ which is a lift of $\sigma$.

Take $S_\ve$ to be a tubular neighborhood of $S$ in $M$ disjoint from $\sigma$. Denote $\tilde S_\ve=\pi^{-1}(S_\ve)$. Since $\tilde S$ is contractible, we have $H_i(\tilde S_\ve)=0$ for all $i\geq 1$. Recall that $\tilde M$ is contractible as well. The composed map
$$H_1(\ti M \setminus \ti S) \to  H_1(\ti M) \to H_1(\tilde M,\tilde S_{\ve})$$
is the zero map. Moreover, from the exact sequence
$$
H_i(\tilde M)\to H_i(\tilde M,\tilde S_\ve)\to H_{i-1}(\tilde S_\ve)
$$
we conclude $H_i(\tilde M,\tilde S_\ve)=0$ for all $i\geq 2$.

Let $\rho:M\setminus S\to [0,+\infty)$ denote a proper smooth function such that $\rho\equiv 0$ in $M\setminus S_\ve$ and $\Lip\rho <1$. We define
$$
L(s)=\frac{1}{2}\min_{\{\rho\leq s\}} R(g).
$$
Fix $s_0=1$ and let $K(s_0)$ be the constant from Proposition \ref{Prop: interior}.

Similar as in \cite[Section 3.1]{CCZ23} for any large constant $\tilde L$ we can construct a piecewisely smooth hypersurface $M_3$ in $\tilde M\setminus \tilde S$ with non-empty boundary such that
\begin{itemize}\setlength{\itemsep}{1mm}
\item $M_3$ has non-zero intersection number with $\tilde \sigma$;
\item $\dist(\partial M_3,\tilde \sigma)\geq \tilde L$;
\item $\partial M_3$ has $\mathbb{T}^*$-stabilized scalar curvature $R_3\geq R(\tilde g)|_{\partial M_3}-c\tilde L^{-2}$, where $c$ is an absolute positive constant independent of $\tilde L$.
\end{itemize}
The independence of $c$ on $\tilde L$ follows from the proof of \cite[Lemma 3.3]{CCZ23}. Actually we can take $c$ to be any positive number greater than $4 \pi^2$.

Take $\tilde L$ large enough so that
$$
c\tilde L^{-2}< \min_{[0,K(s_0)]}L = L(K(s_0)).
$$
Let $\Sigma$ be any component of $\partial M_3$. It is clear that
$$
\Sigma = \Sigma\cap (\tilde M\setminus \tilde S_\ve) \mbox{ modulo }\tilde S_\ve.
$$

Next we show how to establish the diameter estimate for each component of $\Sigma\cap (\tilde M\setminus \tilde S_\ve)$. There are three cases:

\medskip

{\it Case 1. $\Sigma\cap (\tilde M\setminus \tilde S_\ve)=\emptyset$.} This case is trivial.

\medskip

{\it Case 2. $\Sigma\cap (\tilde M\setminus \tilde S_\ve)\neq \emptyset$ and $\Sigma\subset \{\rho\circ \pi\leq K(s_0)\}$.} Then we know that $\Sigma$ has $\mathbb{T}^*$-stabilized scalar curvature
$$
R_3\geq R(\tilde g)|_{\partial M_3}-c\tilde L^{-2}\geq L(K(s_0)).
$$
and so it follows from \cite{Gro20} that its diameter is no greater than
$$
D_1:=2\pi L(K(s_0))^{-\frac{1}{2}}.
$$

\medskip

{\it Case 3. $\Sigma\cap (\tilde M\setminus \tilde S_\ve)\neq \emptyset$ and $\Sigma\cap \{\rho\circ \pi > K(s_0)\}\neq \emptyset$.} By the choice of $\rho$, $\Sigma\cap (\tilde M\setminus \tilde S_\ve)$ is contained in $\Sigma_{s_0} =\Sigma_1$. We can apply Proposition \ref{Prop: interior} to $\Sigma$ associated with the map $(\rho\circ\pi)|_\Sigma$ to obtain that each component of $\Sigma\cap (\tilde M\setminus \tilde S_\ve)$ has diameter no greater than
$$
D_2:=4\pi L(K(s_0))^{-\frac{1}{2}}.
$$

\medskip

From $H_2(\tilde M,\tilde S_\ve)=0$ and the above discussion, we conclude that each component of $\Sigma\cap (\tilde M\setminus \tilde S_\ve)$ is a relative boundary and has diameter no greater than $\max\{D_{1},D_{2}\}$. Using the quantitative filling lemma \cite[Lemma 2.8]{CCZ23} we conclude that there is a $3$-chain $\Gamma$ lying in $D_3$-neighborhood of $\partial M_3$ such that
$$
\partial\Gamma=\Sigma\mbox{ modulo }\tilde S_\ve,
$$
where $D_3$ is an absolute constant depending only on $D_1$, $D_2$ and $(M,g)$ but independent of $\tilde L$. When $\tilde L>D_3$ the relative cycle $M_3+\Gamma$ has non-zero intersection number with $\tilde \sigma$ and so it contributes to some non-trivial element in $H_3(\tilde M,\tilde S_\ve)$, which contradicts the fact $H_3(\tilde M,\tilde S_\ve)=0$. This contradiction shows that $M\setminus S$ cannot admit a complete metric with positive scalar curvature.

\medskip

Lastly, we prove the rigidity part. Suppose that $M\setminus S$ admits a complete metric $g$ of nonnegative scalar curvature. Again by passing to the orientable cover we can assume that $M$ is orientable. By a result of Kazdan \cite{Kaz82}, if $(M\setminus S,g)$ is not Ricci-flat, then there exists a complete metric with positive scalar curvature. Thus $(M\setminus S,g)$ is Ricci-flat.

If $S$ contains any component $S_0^k$ of dimension $k \le n-2$ that is not a hypersurface, then we claim that $S$ contains no hypersurface. Otherwise, we can find a component $U_0$ of $M\setminus S$ such that $U_0$ has at least two ends, for which one is compactified by $S_0$ in $M$ and the other is compactified by some hypersurface $S_1$ in $M$. Then it follows from the Cheeger--Gromoll splitting theorem \cite[Theorem 4]{CG71} that $U_0$ is isometric to some Riemannian product $N\times \mathbb R$, where $N$ is a closed Riemannian manifold. Denote the orientable cover of $S_1$ by $\check S_1$ and the unit tangent sphere bundle of $S_0$ in $M$ by $T_1S_0$. Realize $T_1S_0$ to be the boundary of a tubular neighborhood of $S_0$ in $M$.
Then $\check S_1$ and $T_1S_0$ are homotopy equivalent since they are both homotopy equivalent to $N$.
However, by considering the long exact sequence of homotopy groups for fibration, we have
\[
\pi_{n-k+1}(S_0) \to \pi_{n-k}(\mathbb{S}^{n-k}) \to \pi_{n-k}(T_1S_0).
\]
Since $n-k+1 \ge 3$, and $S_0$ is aspherical, we have $\pi_{n-k+1}(S_0) = 0$, so $\mathbb Z \cong \pi_{n-k}(\mathbb{S}^{n-k}) \to \pi_{n-k}(T_1S_0)$ is injective. But this shows $0 \neq \pi_{n-k}(T_1S_0) = \pi_{n-k}(\check S_1)$, which contradicts that $\check S_1$ is aspherical.

By the homotopy classification of aspherical spaces \cite[Theorem 1.1]{Luc08}, since $S_0$ is not
homotopy equivalent to any finite cover of $M$, $\pi_1(S_0)$ must be a proper subgroup of $\pi_1(M)$ with infinite index. We can consider the covering $\hat\pi: \hat M \to M$ such that $\pi_1(\hat M) = \pi_1(S_0)$. Let $\hat S = \hat\pi^{-1}(S)$. Then we can lift the Ricci-flat metric $g$ to $\hat M \setminus \hat S$ as $\hat g$, and $\hat \pi^{-1}(S_0)$ consists of infinitely many components diffeomorphic to $S_0$, and each of them gives one end in $(\hat M \setminus \hat S, \hat g)$. Recall that $S$ contains no hypersurface and so does $\hat S$, which means that $\hat M\setminus \hat S$ is connected. However, by the Cheeger--Gromoll splitting theorem, a complete connected Ricci-flat manifold has at most two ends, which leads to a contradiction.

We have shown that $S$ consists of hypersurfaces  only, whose collection will be denoted by $\{S_j\}$. For each connected component $U$ of $M\setminus S$, we know that it can have at most two ends. Fix an arbitrary metric on $M$ and denote the metric completion of $U$ by $\overline U$.

If $U$ has two ends, then $ U$ is isometric to a cylinder $N\times \mathbb R$ for some closed Ricci-flat manifold $N$ by the Cheeger--Gromoll splitting theorem \cite[Theorem 4]{CG71}, and each component of the boundary $\partial\overline U$ is the orientable cover of some $S_j$.

If $U$ has only one end, we claim that we can find a two-sheeted cover $\overline U^*$ of $\overline U$ such that the interior of $\overline U^*$ is a cylinder $N\times \mathbb R$ for a closed manifold and the boundary $\partial\overline U^*$ consists of two copies of the orientable cover of some $S_j$. To see this, we start with a careful analysis of the topology of $\overline U$. Since $U$ has only one end,
the boundary $\partial \overline U$ is connected and it is the orientable cover $\check S_j$ of some $S_j$. Recall that $S_j$ is incompressible in $M$, then we can conclude that $\partial \overline U$ is incompressible in $\overline U$, which yields that the map
$\pi_1(\partial \overline U) \to \pi_1(\overline{U})$
is injective. Let us show that $$\pi_1(\partial \overline U) \neq \pi_1(\overline{U}).$$ Otherwise, $i_*:\pi_1(\partial \overline U)\to \pi_1(\overline U)$ is an isomorphism. Recall that $\partial \overline U=\check S_j$ is an Eilenberg--MacLane space, then we can find a continuous map
$$f: \overline U \to \partial \overline U$$ such that $f_*:\pi_1(\overline U)\to \pi_1(\partial \overline U)$ is the inverse map of $i_*:\pi_1(\partial \overline U)\to \pi_1(\overline U)$. Since $f\circ i$ induces an identity map on $\pi_1(\partial \overline U)$, we have the homology relation
$$f_*i_*([\partial\overline U])=[\partial \overline U].$$
However, this is impossible since $\partial \overline U$ is null-homologous in $\overline U$. Therefore, We have shown that $\pi_1(\partial \overline U)$ is a proper subgroup of $\pi_1(\overline U)$. In particular, we can find a cover $\overline U^*$ of $\overline U$ with $\pi_1(\overline U^*)=\pi_1(\partial \overline U)$ such that $\partial \overline U^*$ consists of multiple copies of $\partial \overline U$. Correspondingly, the interior of $\overline U^*$ has multiple ends, and it admits a complete Ricci-flat metric as a cover of $U$. Then the Cheeger--Gromoll splitting theorem yields that the interior of $\overline U^*$ is isometric to a cylinder $N\times \mathbb R$ for some closed Ricci-flat manifold $N$ and so $\partial \overline U^*$ consists of two copies of $\check S_j$.

Thus by lifting to a finite cover of the pair $(M,S)$, we can assume that each connected component $U$ of $M\setminus S$ is isometric to a cylinder $N\times \mathbb R$ for some closed Ricci-flat manifold $N$ and that all $S_j$ are orientable. Let us label the components of $M\setminus S$ by $U_k$ and denote $U_k=N_k\times \mathbb R$. From the cylindrical structure of $U_k$ we see that each $N_k$ admits a non-zero degree map to some closed aspherical manifold $S_j$. It follows from \cite[Proposition 4.4]{CCZ23} that $N_k$ is a closed flat manifold and so the metric $g$ is flat. From the cylindrical structure of $U_k$ we also have the isomorphisms
$$\pi_1(S_j)=\pi_1(U_k)=\pi_1(N_k)$$
through homotopies in $M$. By lifting further to a finite cover of the pair $(M,S)$, we can assume $\pi_1(N_k)=\mathbb Z^{n-1}$ for all $k$. From the second Bieberbach rigidity theorem we know that each $N_k$ is diffeomorphic to $\mathbb T^{n-1}$, and so all $U_k$ are cylinders $\mathbb T^{n-1}\times \mathbb R$. We complete the proof.
\end{proof}

\section{Proof of Theorem \ref{Thm: 5d}}\label{Sec: proof of 5d}

\begin{proof}[Proof of Theorem \ref{Thm: 5d}]
Let $n$ be the dimension of $M$. Here $n =5$ but the same proof also works if $n$ equals 3 or 4. We note that the rigidity part is exactly the same as in the proof of Theorem \ref{Thm: main}. Thus we only need to show that $M\setminus S$ cannot admit any complete metric with positive scalar curvature. According to our assumption, there are two cases.

\medskip

{\it Case 1. $S$ contains a hypersurface $\Sigma^{n-1}$ of codimension $1$.}
By passing to the orientable cover of $M$, we can assume that $M$ is orientable. Denote the orientable cover of $\Sigma$ by $\check \Sigma$ and we can realize it as a boundary component of some small tubular neighborhood of $\Sigma$ in $M$. Clearly, $\check \Sigma$ is orientable and incompressible in $M$ and so in $M\setminus S$.
Since $\check\Sigma$ is aspherical, by \cite[Theorem 1.1]{CLSZ21} $M \setminus S$ cannot admit any complete metric with positive scalar curvature.

\medskip

{\it Case 2. $S$ contains a submanifold $\Sigma^{n-2}$ of codimension $2$.} Again we can assume that $M$ is orientable. We make a discussion depending on whether $\Sigma$ is orientable or not.

Suppose that $\Sigma$ is orientable.
Take $\Sigma_\ve$ to be a small tubular neighborhood of $\Sigma$ in $M$. Then $\partial \Sigma_\ve$ is a circle bundle over $\Sigma$. By considering the long exact sequence of homotopy groups for fibration, we see that $\partial \Sigma_\ve$ is again an aspherical manifold, and $\pi_1(\partial \Sigma_\ve)$ is generated by curves in $\pi_1(\partial D)$ and $\pi_1(\Sigma)$, where $D$ is the disk fiber of $\Sigma_\ve$ over a point $p \in \Sigma$.
Our goal is to show that $\pi_1(\partial \Sigma_\ve) \to \pi_1(M\setminus S)$ is injective.  Then it follows from \cite[Theorem 1.1]{CLSZ21} that $M \setminus S$ cannot admit any complete metric with positive scalar curvature.

Orient $D$ such that the intersection number of $D$ and $\Sigma$ in $M$ is $1$. Let us consider the map $\pi_1(\partial D) \to \pi_1(M\setminus S)$ induced by the inclusion. We claim that the map $\pi_1(\partial D) \to \pi_1(M\setminus S)$ is injective. Otherwise,
an element $k[\partial D]$ for some $k \in \mathbb{Z}$ is mapped to zero in $\pi_1(M\setminus S)$. In other words, it bounds a disk $D'$ in $M\setminus S$. On the other hand, $k[\partial D]$ bounds a disk $D''$ in $\Sigma_\ve$ having intersection number $k$ with $\Sigma$.
Then $D' \cup D''$ is a sphere in $M$ which has intersection number $k$ with $\Sigma$. However, since $M$ is aspherical, we have $\pi_2(M) =0$. As a consequence, the sphere $D'\cup D''$ must be trivial in $H_2(M)$ and so $k =0$. Thus the map $\pi_1(\partial D) \to \pi_1(M\setminus S)$ is injective.

Now we consider the inclusion maps
$$i_1:\partial\Sigma_\ve\to M\setminus S\mbox{ and }i_2:M\setminus S\to M.$$
To see the injectivity of the map $\pi_1(\partial\Sigma_\ve)\to \pi_1(M\setminus S)$, we take an element $\gamma$ in $\pi_1(\partial\Sigma_\ve)$ such that $(i_1)_*\gamma=0$ in $\pi_1(M\setminus S)$, and try to show $\gamma=0$ in $\pi_1(\partial\Sigma_\ve)$. It is clear that we have $(i_2\circ i_1)_*\gamma=0$ in $\pi_1(M)$. Since $\Sigma$ is incompressible, we conclude that $\pi_*\gamma=0$ in $\pi_1(\Sigma)$, where $\pi:\partial\Sigma_\ve\to \Sigma$ is the projection map of the circle bundle. From the long exact sequence of homotopy groups for fibration, the sequence
$$\pi_1(\partial D)\to\pi_1(\partial\Sigma_\ve)\to \pi_1(\Sigma)$$
is exact. As a result, we can find an element $\xi$ in $\pi_1(\partial D)$ such that $i_*\xi=\gamma$ in $\pi_1(\partial\Sigma_\ve)$, where $i:\partial D\to \partial\Sigma_\ve$ is the inclusion map of the circle bundle. In particular, we have $(i_1)_*i_*\xi =0$ in $\pi_1(M\setminus S)$. Then we can conclude from the injectivity of the map $\pi_1(\partial D)\to\pi_1(M\setminus S)$ that $\xi=0$ in $\pi_1(\partial D)$ and so $\gamma=0$ in $\pi_1(\partial\Sigma_\ve)$.

Suppose that $\Sigma$ is not orientable. Denote the orientable cover of $\Sigma$ by $\check\Sigma$. Then we can find a cover $(\check M,\check S)$ of the pair $(M,S)$ such that $\pi_1(\check M)=\pi_1(\check\Sigma)$. In particular, the subset $\check S$ contains $\check\Sigma$ as a component. Similarly, we can show $\pi_1(\partial \check\Sigma_\ve) \to \pi_1(\check M \setminus \check S)$ is injective. This implies that $\check M\setminus\check S$ admits no complete metric with positive scalar curvature and so does $M\setminus S$.

\end{proof}

\end{document}